\documentclass[nosumlimits,twoside]{amsart}
\usepackage{amssymb}
\usepackage{amsmath}
\usepackage[bookmarksnumbered,colorlinks,plainpages,hyperindex]{hyperref}
\usepackage{amsfonts}
\usepackage{color}
\usepackage{float}
\usepackage{graphicx}
\usepackage{tikz}
\usepackage[hmargin=2cm,vmargin=3.5cm]{geometry}

\usetikzlibrary{matrix,arrows}
\newtheorem{theorem}{\sc Theorem}[section]
\newtheorem{proposition}[theorem]{\sc Proposition}
\newtheorem{lemma}[theorem]{\sc Lemma}
\newtheorem{corollary}[theorem]{\sc Corollary}
\theoremstyle{definition}
\newtheorem{definition}[theorem]{\sc Definition}

\theoremstyle{remark}
\newtheorem{remark}[theorem]{\sc Remark}

\newtheorem{claim}[theorem]{}

\input{tcilatex}

\begin{document}
\title{Separable $K$-Linear Categories}
\author{Andrei Chite\c{s} }
\thanks{The first author was financially supported by CNCSIS (Contract
560/2009, CNCSIS code ID\_69).}
\address{University of Bucharest, Faculty of Mathematics and Informatics,
The Research Group in Geometry, Topology and Algebra, Str. Academiei 14,
Ro-70109, Bucharest, Romania}
\email{andrei.chites@pointlogistix.ro}
\author{Costel Chite\c{s} }
\address{Tudor Vianu High School, Str. Arhitect Ion Mincu 10, Bucharest,
Romania}
\email{costelchites@yahoo.com}
\date{}

\begin{abstract}
We define and investigate separable $K$-linear categories. We show that such
a category $\mathcal{C}$ is locally finite and that every left $\mathcal{C}$%
-module is projective. We apply our main results to characterize separable
linear categories that are spanned by groupoids or delta categories.
\end{abstract}

\maketitle

\section*{Introduction}

Linear categories are important generalizations of ordinary associative
algebras, that play an important role in various fields of mathematics, such
as representation theory of finite dimensional algebras, mathematical
physics, etc. They were introduced and studied in \cite{M1}, while in \cite%
{M2} several homological tools were adapted to this more general framework.
In particular, in loc. cit. Hochschild-Mitchell cohomology was defined as a
substitute of Hochschild cohomology, which is a key homological invariant of
unital associative algebras.

The aim of this short note is to investigate the basic properties of the
simplest linear categories from a cohomological point of view. More
precisely, we give equivalent characterizations of those linear categories
with the property that their Hochschild-Mitchell cohomology groups vanish in
positive degrees, see Theorem \ref{pr:caracterizare}. It is worthwhile to
remark that for associative algebras a similar result can be found in \cite%
{We}, and in \cite{St} in the more general case of algebras in an abelian
monoidal category. In analogy to the case of associative algebras, we call
these linear categories separable. We also show that a separable linear
category $\mathcal{C}$ is locally finite, i.e. $\dim _{K}\mathrm{Hom}_{%
\mathcal{C}}(x,y)<\infty ,$ for any objects $x$ and $y$ in $\mathcal{C}.$
The later result may be seen as a generalization of Zelinsky Theorem. As
applications of our main results we give necessary and sufficient conditions
such that the $K$-linear categories spanned by groupoids and delta
categories to be separable.

\section{Preliminaries}

Throughout this paper $\mathcal{C}$ denotes a small category. The set of
objects in $\mathcal{C}$ is denoted by $\mathcal{C}_{0}$ and, for
simplicity, we write $\mathcal{C}(x,y)$ for $\mathrm{Hom}_{\mathcal{C}}(x,y)$%
.

\begin{claim}
$K$\textbf{-linear categories.} Let $K$ be a field. A category $\mathcal{C}$
is said to be $K$-linear if $\mathcal{C}(x,y)$ is a $K$-vector space, for
any $x,y\in \mathcal{C}_{0},$ and the composition maps in $\mathcal{C}$ are
bilinear. Note that the composition in $\mathcal{C}$ can be seen as linear
maps%
\begin{equation*}
\mathcal{C}(y,z)\otimes \mathcal{C}(x,y)\rightarrow \mathcal{C}(x,z),\qquad
g\otimes f\mapsto g\circ f.
\end{equation*}%
For the properties of linear categories the reader is refered to \cite{HS,
M2} and the references therein.

Let $\mathcal{C}$ and $\mathcal{D}$ be two $K$-linear categories. A functor $%
F:\mathcal{C}\rightarrow \mathcal{D}$ is said to be $K$-linear if $F(-):%
\mathcal{C}(x,y)\rightarrow \mathcal{D}(F(x),F(y))$ is a $K$-linear map, for
all $x,y\in \mathcal{C}_{0}.$
\end{claim}

\begin{claim}
\textbf{Modules and bimodules over }$K$\textbf{-linear categories.} A left
module over $\mathcal{C}$ is a $K$-linear functor $M:\mathcal{C}\rightarrow
K $-$\mathrm{Mod.}$ Note that $M$ is defined by a family of vector spaces $%
({}_{x}M)_{x\in \mathcal{C}_{0}}$ and $K$-linear maps%
\begin{equation*}
\rhd :\mathcal{C}(y,x)\otimes {}_{x}M\rightarrow {}_{y}M
\end{equation*}%
satisfying identities that are similar to that ones that appear in the
definition of modules over associative algebras. A module morphism $%
f:M\rightarrow N$ is a natural transformation between the functors $M$ and $%
N.$ It is given by a family $({}_{x}{}f)_{x\in \mathcal{C}_{0}}$ of $K$%
-linear maps $_{x}f:{}_{x}M\rightarrow {}_{x}N,$ which are also linear with
respect to the structure maps that define the module structures on $M$ and $%
N $. Right $\mathcal{C}$-modules are defined analogously. We obtain two
categories $\mathcal{C}$-$\mathrm{Mod}$ and $\mathrm{Mod}$-$\mathcal{C}.$

To define $\mathcal{C}$-bimodules one defines a new linear category $%
\mathcal{C}\boxtimes _{K}\mathcal{C}^{op},$ see \cite{HS} for details. By
definition, a $\mathcal{C}$-bimodule is a left $\mathcal{C}\boxtimes _{K}%
\mathcal{C}^{op}$-module, that is a family $({}_{x}M_{y})_{(x,y)\in \mathcal{%
C}_{0}\times \mathcal{C}_{0}}$ together with left and right actions%
\begin{eqnarray*}
\rhd &:&\mathcal{C}(y,x)\otimes {}_{x}M_{z}\rightarrow {}_{y}M_{z} \\
\lhd &:&{}_{z}M_{x}\otimes \mathcal{C}(y,x){}\rightarrow {}_{z}M_{y}
\end{eqnarray*}%
such that, for all $x_{0}$ and $y_{0}$ in $\mathcal{C}_{0},$ the pairs $%
\left( ({}_{x}M_{y_{0}})_{x\in \mathcal{C}_{0}},\rhd \right) $ and $\left(
({}_{x_{0}}M_{y})_{y\in \mathcal{C}_{0}},\lhd \right) $ are a left and a
right $\mathcal{C}$-module, respectively, and these structures are
compatible in the obvious sense. We shall denote these modules by $%
_{-}M_{y_{0}}$ and $_{x_{0}}M_{-},$ respectively. The category of $\mathcal{C%
}$-bimodules is denoted by $\mathcal{C}$-$\mathrm{Mod}$-$\mathcal{C}.$ Note
that $\mathcal{C}$ can be seen in a canonical way as an object in $\mathcal{C%
}$-$\mathrm{Mod}$-$\mathcal{C}.$

The categories $\mathcal{C}$-$\mathrm{Mod,}$ $\mathrm{Mod}$-$\mathcal{C}$
and $\mathcal{C}$-$\mathrm{Mod}$-$\mathcal{C}$ are abelian and have enough
projective and injective objects, cf. \cite{HS}. Thus we may consider
\textrm{Ext }functors in these categories.
\end{claim}

\begin{claim}
\textbf{Hochschild-Mitchell cohomology. }Let $\mathcal{C}$ be a $K$-linear
category, and let $M$ be a $\mathcal{C}$-bimodule. Hochschild-Mitchell
cohomology of $\mathcal{C}$ with coefficients in $M$ is defined by%
\begin{equation*}
H^{\ast }(\mathcal{C},M):=\mathrm{Ext}_{\mathcal{C}-\mathcal{C}}^{\ast }(%
\mathcal{C},M),
\end{equation*}%
where $\mathrm{Ext}_{\mathcal{C}-\mathcal{C}}^{\ast }(\mathcal{-},-)$ denote
the $\mathrm{Ext}$ functors in the category $\mathcal{C}$-$\mathrm{Mod}$-$%
\mathcal{C}.$

Many of the properties of Hochschild-Mitchell cohomology follow immediately
from the fact that this cohomology theory is defined using derived functors
in an abelian category. The most important ones for our work are the
following.

First, if $n\in \mathbb{N}^{\ast }$ and $M$ is an injective $\mathcal{C}$%
-bimodule, then $H^{n}(\mathcal{C},M)=0$. This equality also holds if $%
\mathcal{C}$ is a projective as a $\mathcal{C}$-bimodule.

Second, if $0\rightarrow M\rightarrow N\rightarrow P\rightarrow 0$ is an
exact sequence of $\mathcal{C}$-bimodules, then the exact sequence of the
\textrm{Ext}\textit{\ }functor, applied to the above short exact sequence,
yields the following long exact sequence:
\begin{align}
0\rightarrow & H^{0}(\mathcal{C},M)\rightarrow H^{0}(\mathcal{C}%
,N)\rightarrow H^{0}(\mathcal{C},P)\rightarrow H^{1}(\mathcal{C}%
,M)\rightarrow H^{1}(\mathcal{C},N)\rightarrow H^{1}(\mathcal{C}%
,P)\rightarrow ...  \label{eq:leq} \\
\rightarrow & H^{n}(\mathcal{C},M)\rightarrow H^{n}(\mathcal{C}%
,N)\rightarrow H^{n}(\mathcal{C},P)\rightarrow H^{n+1}(\mathcal{C}%
,M)\rightarrow H^{n+1}(\mathcal{C},N)\rightarrow H^{n+1}(\mathcal{C}%
,P)\rightarrow ...  \notag
\end{align}
\end{claim}

\section{Separable linear categories}

We are going to study the $K$-linear categories that are simple from a
cohomological point of view. More exactly, we are going to study the
properties of a $K$-linear category $\mathcal{C}$ such that its Hochschild
cohomology in positive degrees is trivial.

\begin{definition}
A $K$-linear category $\mathcal{C}$ is separable if $H^{1}(\mathcal{C},M)=0,$
for any $\mathcal{C}$-bimodule $M$.
\end{definition}

\begin{lemma}
The bimodule $\mathcal{C}\otimes \mathcal{C}$ is projective (i.e. it is a
projective object in $\mathcal{C}$-$\mathrm{Mod}$-$\mathcal{C}$).
\end{lemma}

\begin{proof}
Consider the following diagram
\begin{equation*}
\begin{tikzpicture}[description/.style={fill=white,inner sep=2pt}] \matrix
(m) [matrix of math nodes, row sep=3em, column sep=2.5em, text height=1.5ex,
text depth=0.25ex] { &&\mathcal{C}\otimes\mathcal{C} \\ & M & N & 0\\ };
\path[->,font=\scriptsize] (m-2-3) edge node[auto] {$ $} (m-2-4) (m-2-2)
edge node[auto] {$ \pi $} (m-2-3) (m-1-3) edge node[auto] {$ \varphi $}
(m-2-3); \end{tikzpicture}
\end{equation*}%
\noindent where $\varphi $ and $\pi $ are arbitrary bimodule morphisms with $%
\pi $ epimorphism. By definition we have $_{x}(\mathcal{C}\otimes \mathcal{C}%
)_{y}=\bigoplus_{z}\mathcal{C}(z,x)\otimes \mathcal{C}(y,z).$ Since $\pi $
is an epimorphism, for every $z\in \mathcal{C}_{0},$ there is $_{z}m_{z}\in
{}_{z}M_{z}$ such that $_{z}\pi _{z}(_{z}m_{z})=$ $_{z}\varphi
_{z}(1_{z}\otimes 1_{z}).$ {}We define%
\begin{equation*}
_{x}\psi _{y}^{z}:{}\mathcal{C}(z,x)\otimes \mathcal{C}(y,z)\rightarrow
{}_{x}M_{y},\qquad _{x}\psi _{y}(f\otimes g)=f\rhd {}_{z}m_{z}\lhd g.
\end{equation*}
Let $_{x}\psi _{y}:{}_{x}(\mathcal{C}\otimes \mathcal{C})_{y}\rightarrow
{}_{x}M_{y}$ be the $K$-linear map induced by the family $\left( _{x}\psi
_{y}^{z}\right) _{z\in \mathcal{C}_{0}}.$ It is easy to see that the family $%
\left( _{x}\psi _{y}\right) _{x,y\in \mathcal{C}_{0}}\ $is a morphism of
bimodules such that $\pi \circ \psi =\varphi $. Thus $\mathcal{C}\otimes
\mathcal{C}$ is projective.
\end{proof}

\begin{theorem}
\label{pr:caracterizare} Let $\mathcal{C}$ be a K-linear category. The
following statements are equivalent:

\begin{enumerate}
\item $\mathcal{C}$ is separable.

\item $H^{n}(\mathcal{C},M)=0$ for all $n>0$ and $M\in \mathcal{C}$-$\mathrm{%
Mod}$-$\mathcal{C}.$

\item $\mathcal{C}$ is projective as a bimodule.

\item The canonical morphism $\mathcal{C}\otimes \mathcal{C}\rightarrow
\mathcal{C}$ splits in the category $\mathcal{C}$-$\mathrm{Mod}$-$\mathcal{C}%
.$

\item There is a family $(a_{x}^{y})_{x,y\in \mathcal{C}_{0}}$ with the
following properties:

\begin{enumerate}
\item The element $a_{x}^{y}\in \mathcal{C}(y,x)\otimes \mathcal{C}(x,y),\ $%
for all $x,y\in \mathcal{C}_{0}$.

\item For any $x\in \mathcal{C}_{0},$ the family $(a_{x}^{y})_{y\in \mathcal{%
C}_{0}}$ is of finite support.

\item For every object $x,$ we have $\sum_{y\in \mathcal{C}_{0}}\mathrm{comp}%
(a_{x}^{y})=1_{x}.$

\item If $f\in C(x,z),$ then $f\vartriangleright
a_{x}^{y}=a_{z}^{y}\vartriangleleft f$.
\end{enumerate}
\end{enumerate}
\end{theorem}

\begin{proof}
$(1)\Leftrightarrow (2)$ Recall that, by definition, $\mathcal{C}$ is
separable if and only if $H^{1}(\mathcal{C},M)=0$, for every bimodule M.
Therefore, $(2)\Rightarrow (1)$ is straightforward. The other implication
can be proved by induction as follows. We assume that $H^{n}(\mathcal{C}%
,Q)=0,\ $for all $n>0$ and $Q\in \mathcal{C}$-$\mathrm{Mod}$-$\mathcal{C}.$
As in the category of $\mathcal{C}$-bicomodules there are enough injective
objects, there exists an injective morphism $i:M\rightarrow I$ in $\mathcal{C%
}$-$\mathrm{Mod}$\textrm{-}$\mathcal{C}$. Let $Q$ be the cokernel of $i,$ so
the following sequence is exact in $\mathcal{C}$-$\mathrm{Mod}$\textrm{-}$%
\mathcal{C}$.
\begin{equation*}
0\longrightarrow M\longrightarrow I\longrightarrow Q\longrightarrow 0.
\end{equation*}%
From (\ref{eq:leq}) we get the following exact sequence:%
\begin{equation*}
H^{n}(\mathcal{C},Q)\rightarrow H^{n+1}(\mathcal{C},M)\rightarrow H^{n+1}(%
\mathcal{C},I)
\end{equation*}%
By induction hypothesis, $H^{n}(\mathcal{C},Q)=0.$ On the other hand $%
H^{n+1}(\mathcal{C},I)=0,$ as $I$ is injective as a $\mathcal{C}$-bimodule.
Thus $H^{n+1}(\mathcal{C},M)=0,$ too.

$(2)\Leftrightarrow (3)$ Hochschild-Mitchell cohomology of $\mathcal{C}$
with coefficients in $M$ is defined by $H^{n}(\mathcal{C},M)=\mathrm{Ext}_{%
\mathcal{C}-\mathcal{C}}^{n}(C,M).$ Furthermore, an object $X$ in an abelian
category $\mathcal{A}$ is projective if, and only if, $\mathrm{Ext}_{%
\mathcal{A}}^{n}(X,Y)=0,$ for all $n>0$ and $Y\in \mathrm{Ob}(\mathcal{A}).$
Thus, $\mathcal{C}$ is projective as a $\mathcal{C}$-bimodule if and only if
$H^{n}(\mathcal{C},M)=0,$ for all $n>0\ $and $M\in \mathcal{C}$-$\mathrm{Mod}
$-$\mathcal{C}$.

$(3)\Rightarrow (4)$ Consider the following diagram in $\mathcal{C}$-$%
\mathrm{Mod}$-$\mathcal{C}:$%
\begin{equation*}
\begin{tikzpicture}[description/.style={fill=white,inner sep=2pt}] \matrix
(m) [matrix of math nodes, row sep=3em, column sep=2.5em, text height=1.5ex,
text depth=0.25ex] { && C\\ & \mathcal{C}\otimes\mathcal{C} & \mathcal{C}\\
}; \path[->,font=\scriptsize] (m-2-2) edge node[auto,swap] {$ \mathrm{comp}
$} (m-2-3) (m-1-3) edge node[auto] {$ 1_{\mathcal{C}} $} (m-2-3);
\end{tikzpicture}
\end{equation*}%
If $\mathcal{C}$ is projective, then there exists a morphism of $\mathcal{C}$%
-bimodules $\varphi :\mathcal{C}\rightarrow \mathcal{C}\otimes \mathcal{C}$
such that $\mathrm{comp}\circ \varphi =1_{C}$. This proves that $\varphi $
is a section of $\mathrm{comp,}$ that is $\mathrm{comp}$ splits in $\mathcal{%
C}$-$\mathrm{Mod}$-$\mathcal{C}$.

$(4)\Rightarrow (3)$ Suppose that the canonical morphism $\mathrm{comp}:%
\mathcal{C}\otimes \mathcal{C}\rightarrow \mathcal{C}$ has a section in $%
\mathcal{C}$-$\mathrm{Mod}$-$\mathcal{C}$. It results that $M,$ the kernel
of $\mathrm{comp}$, is a complement of $\mathcal{C}$ in $\mathcal{C}\otimes
\mathcal{C}$. From the above lemma $\mathcal{C}\otimes \mathcal{C}$ is
projective, so $\mathcal{C}$ is a projective $\mathcal{C}$-bimodule, since
it is a direct summand in a projective bimodule.

$(4)\Rightarrow (5)$ Let $\varphi :\mathcal{C}\rightarrow \mathcal{C}\otimes
\mathcal{C}$ be a $\mathcal{C}$-bimodule morphism such that $\mathrm{comp}%
\circ \varphi =1_{\mathcal{C}}$. We fix $(x,y)\in \mathcal{C}_{0}\times
\mathcal{C}_{0}$. Let $\ {}_{x}\varphi _{y}:\mathcal{C}(y,x)\rightarrow
\bigoplus_{z\in C_{0}}\mathcal{C}(z,x)\otimes \mathcal{C}(y,z)$ be the
corresponding component of $\varphi .$ We have $\ $%
\begin{equation*}
{}_{x}\varphi _{x}(1_{x})\in \bigoplus_{y\in C_{0}}\mathcal{C}(y,x)\otimes
\mathcal{C}(x,y),
\end{equation*}%
so $_{x}\varphi _{x}(1_{x})=(a_{x}^{y})_{y\in C_{0}}$, where $a_{x}^{y}\in
\mathcal{C}(y,x)\otimes \mathcal{C}(x,y)$. Hence the family $%
(a_{x}^{y})_{x,y\in \mathcal{C}_{0}}$ satisfies the first property in $(5)$.
It also satisfies the second property as $_{x}\varphi _{x}(1_{x})$ is an
element in $\bigoplus_{y\in \mathcal{C}_{0}}\mathcal{C}(y,x)\otimes \mathcal{%
C}(x,y)$, so the family $(a_{x}^{y})_{y\in \mathcal{C}_{0}}$ has finite
support. Moreover, $\mathrm{comp}(_{x}\varphi _{x}(1_{x}))=1_{x}$, since $%
\varphi $ is a section of $\mathrm{comp}$. Thus, for all $x\in \mathcal{C}%
_{0}$%
\begin{equation*}
\sum_{y\in \mathcal{C}_{0}}\mathrm{comp}(a_{x}^{y})=1_{x},
\end{equation*}%
i.e. $(a_{x}^{y})_{y\in \mathcal{C}_{0}}$ satisfies the third property. Let
us now show that $(a_{x}^{y})_{x,y\in \mathcal{C}_{0}}$ satisfies $\ $the
last property$.$ Let $f\in \mathcal{C}(x,z)$. Since $a_{x}^{y}\in \mathcal{C}%
(y,x)\otimes \mathcal{C}(x,y),$ we can write this element as a sum%
\begin{equation*}
a_{x}^{y}=\sum_{i=1}^{n_{x,y}}f_{y,x}^{i}\otimes g_{x,y}^{i},
\end{equation*}%
where $f_{y,x}^{i}\in \mathcal{C}(y,x)$ and $g_{x,y}^{i}\in \mathcal{C}(x,y)$%
. Then we get
\begin{equation*}
f\rhd a_{x}^{y}=\sum_{i=1}^{n_{x,y}}f\circ f_{y,x}^{i}\otimes g_{x,y}^{i}%
\text{ and }a_{z}^{y}\lhd f=\sum_{i=1}^{n_{x,y}}f_{y,z}^{i}\otimes
g_{z,y}^{i}\circ f.
\end{equation*}%
On the other hand
\begin{equation*}
_{x}\varphi _{z}(f)={}_{x}\varphi _{z}(f\rhd 1_{x})=f\rhd {}_{x}\varphi
_{x}(1_{x})=f\rhd (a_{x}^{y})_{y\in \mathcal{C}_{0}}=(f\rhd a_{x}^{y})_{y\in
\mathcal{C}_{0}}
\end{equation*}%
and, similarly,
\begin{equation*}
_{x}\varphi _{z}(f)={}_{x}\varphi _{z}(1_{z}\lhd f)={}_{z}\varphi
_{z}(1_{z})\lhd f=(a_{z}^{y})_{y\in \mathcal{C}_{0}}\lhd f=(a_{z}^{y}\lhd
f)_{y\in \mathcal{C}_{0}}.
\end{equation*}%
So $(f\rhd a_{x}^{y})_{y\in C_{0}}=(a_{z}^{y}\lhd f)_{y\in C_{0}}$. We
deduce that $f\rhd a_{x}^{y}=a_{z}^{y}\lhd f,$ for all $y\in \mathcal{C}%
_{0}. $

$(5)\Rightarrow (4)$ Let $(a_{x}^{y})_{x,y\in \mathcal{C}_{0}}$ a family
which satisfies (a)-(d). We define:
\begin{equation*}
_{x}\varphi _{y}:\mathcal{C}(y,x)\rightarrow \bigoplus_{z\in \mathcal{C}_{0}}%
\mathcal{C}(z,x)\otimes \mathcal{C}(y,z),\qquad _{x}\varphi _{y}(f)=(f\rhd
a_{y}^{z})_{z\in \mathcal{C}_{0}}.
\end{equation*}%
The map $_{x}\varphi _{y}$ is well-defined because $(a_{y}^{z})_{z\in
\mathcal{C}_{0}}$ is of finite support. Obviously $\varphi =(_{x}\varphi
_{y})_{x,y\in \mathcal{C}_{0}}$ is a morphism of left $\mathcal{C}$-modules.
The family $\varphi $ defines a morphism of right $\mathcal{C}$-modules
because $(a_{x}^{y})_{x,y\in \mathcal{C}_{0}}$ satisfies (d). Finally,
taking into account (c), $\varphi :\mathcal{C}\rightarrow \mathcal{C}\otimes
\mathcal{C}$ is a section for $\mathrm{comp}:\mathcal{C}\otimes \mathcal{C}%
\rightarrow \mathcal{C}$ in $\mathcal{C}$-$\mathrm{Mod}$-$\mathcal{C}$ $%
\displaystyle.$
\end{proof}

Using the equivalent characterization of separable linear categories in
Theorem \ref{pr:caracterizare} we shall now prove a generalization of
Zelinsky Theorem, which states that a separable algebra is finite
dimensional (as a vector space over the base field).

\begin{definition}
We say that a $K$-linear category $\mathcal{C}$ is locally finite
dimensional if $\dim _{K}\mathcal{C}(x,y)<\infty $ for all $x,y\in \mathcal{C%
}_{0}$.
\end{definition}

\begin{theorem}
A $K$-linear separable category $\mathcal{C}$ is locally finite dimensional.
\end{theorem}

\begin{proof}
Since $\mathcal{C}$ is separable, there is a family $(a_{x}^{y})_{x,y\in
C_{0}}$ that satisfies the properties (a)-(d) in Theorem \ref%
{pr:caracterizare}.(5). We write each $a_{x}^{y}$ as a sum
\begin{equation}
a_{x}^{y}=\sum_{i=1}^{n_{x,y}}f_{y,x}^{i}\otimes g_{x,y}^{i},  \label{eq:axy}
\end{equation}%
with $f_{y,x}\in \mathcal{C}(y,x)$ and $g_{x,y}\in \mathcal{C}(x,y)$. From
all representations of $a_{x}^{y}$ as in (\ref{eq:axy}) we choose one such
that the number $n_{x,y}$ is minimal. For such a choice the set $%
\{g_{x,y}^{i}\mid i=1,...,n_{x,y}\}$ is linearly independent. Thus, for
every $i=1,...,n_{x,y},$ there is a $K$-linear application $\alpha
_{x,y}^{i}:\mathcal{C}(x,y)\rightarrow K$ such that $\alpha
_{x,y}^{i}(g_{x,y}^{j})=\delta _{i,j},$ for any $1\leq j\leq n_{x,y}$. If $%
f\in \mathcal{C}(x,z)$ and $y\in \mathcal{C}_{0}$ then $f\rhd
a_{x}^{y}=a_{z}^{y}\lhd f.\ $ Equivalently, we have the following identity
in $\mathcal{C}(y,z)\otimes \mathcal{C}(x,y)$%
\begin{equation}
\sum_{p=1}^{n_{x,y}}f\circ f_{y,x}^{p}\otimes
g_{x,y}^{p}=\sum_{q=1}^{n_{z,y}}f_{y,z}^{q}\otimes g_{z,y}^{q}\lhd f.
\label{eq:fxy}
\end{equation}%
Let $\mathcal{V}_{y,x}$ denote the vector space generated by $f_{y,x}^{i},$
where $i=1,...,n_{x,y}.$ By construction, $\dim _{K}\mathcal{V}_{y,x}<\infty
$. For a given $i,$ let us apply $1_{\mathcal{C}(y,x)}\otimes \alpha
_{x,y}^{i}$ to the left and right sides of (\ref{eq:fxy}). For $x,y,z\in
\mathcal{C}_{0}$ and $f\in \mathcal{C}(x,z)$, one obtains
\begin{equation*}
f\circ f_{y,x}^{i}=\sum_{q=1}^{n_{z,y}}\alpha _{x,y}^{q}(g_{z,y}^{q}\lhd
f)f_{y,z}^{q}.
\end{equation*}%
This relation shows that $f\circ f_{y,x}^{i}\in \mathcal{V}_{y,z}$.
Furthermore, the composition in $\mathcal{C}$ induces an application
\begin{equation*}
\varphi _{x,y,z}:\mathcal{C}(x,z)\rightarrow \mathrm{Hom}_{K}(\mathcal{V}%
_{y,x},\mathcal{V}_{y,z}),\qquad \varphi _{x,y,z}(f)(g)=f\circ g.
\end{equation*}%
We fix $(x,z)\in \mathcal{C}_{0}\times \mathcal{C}_{0}$. The family $%
(a_{x}^{y})_{y\in \mathcal{C}_{0}}$ is of finite support, so there exist $%
y_{1},...,{y}_{p}\in \mathcal{C}_{0}$ such that $a_{x}^{y}=0,$ for every $y$
which does not belong to $\{y_{1},...,{y}_{p}\}$. We define
\begin{equation*}
\varphi _{x,z}:\mathcal{C}(x,z)\rightarrow \bigoplus_{j=1}^{p}\mathrm{Hom}%
_{K}(\mathcal{V}_{y_{j},x},\mathcal{V}_{y_{j},z}),\qquad \varphi
_{x,z}(f)=\left( \varphi _{x,y_{j},z}(f)\right) _{1\leq j\leq p}.
\end{equation*}%
We claim that $\varphi _{x,z}$ is injective. If $f\in \mathrm{Ker}\varphi
_{x,z}$, then $\varphi _{x,y_{j},z}(f)=0\ $for any $1\leq j\leq p$. Thus $%
\varphi _{x,y_{j},z}(f)(g)=0,\ $for all $g\in \mathcal{V}_{y_{i},x}$. In
particular, by taking $g:=f_{y_{j},x}^{i},$ we get $f\circ f_{y_{j},x}^{i}=0$
for any $1\leq i\leq n_{x,y_{j}}.$ It results that
\begin{equation*}
f\rhd a_{x}^{y_{j}}=\sum_{i=1}^{n_{x,y_{j}}}f\circ f_{y_{j},x}^{i}\otimes
g_{x,y_{j}}^{i}=0,
\end{equation*}%
for all $1\leq j\leq p$. On the other hand, if $y\not\in \{y_{1},...,y_{p}\}$
then $a_{x}^{y}=0$. We deduce that
\begin{equation*}
f=f\rhd 1_{x}=f\rhd \mathrm{comp}(\sum_{j=1}^{p}a_{x}^{y_{j}})=\sum_{j=1}^{p}%
\mathrm{comp}(f\rhd a_{x}^{y_{j}})=0.
\end{equation*}%
In conclusion, $\varphi _{x,z}$ is injective, as we claimed. Therefore, $%
\mathcal{C}(x,z)$ can be embedded in the vector space $\mathcal{V}%
=\bigoplus_{j=1}^{p}\mathrm{Hom}_{K}(\mathcal{V}_{y_{j},x},\mathcal{V}%
_{y_{j},z})$. Note that $\mathcal{V}$ is a finite dimensional vector space,
being a finite direct sum of finite dimensional vector spaces. Thus, $%
\mathcal{C}(x,z)\ $is obviously finite dimensional, for every $x,z\in
\mathcal{C}_{0}.$
\end{proof}

Let $\mathcal{A}$ be a not necessarily linear category. The $K$%
-linearization of $\mathcal{A}$ is the $K$-linear category $K\left[ \mathcal{%
A}\right] $ that has the same objects as $\mathcal{A}$, but
\begin{equation*}
K\left[ \mathcal{A}\right] \left( x,y\right) :=\left\langle f\mid f\in
\mathcal{A}(x,y)\right\rangle _{K}.
\end{equation*}%
Therefore, by definition, $K\left[ \mathcal{A}\right] \left( x,y\right) $ is
the $K$-vector space having $\mathcal{A}(x,y)$ as a basis. The composition
in $K\left[ \mathcal{A}\right] $ is the unique bilinear extension of the
composition in $\mathcal{A}.$

Recall that $\mathcal{G}$ is a groupoid if, and only if, all morphisms in $%
\mathcal{G}$ are invertible. We can now prove the following corollary, that
generalizes Maschke Theorem from group algebras.

\begin{corollary}
Let $\mathcal{G}$ be a small groupoid. Then $K\left[ \mathcal{G}\right] $ is
separable if, and only if, $\mathcal{G}(x,y)$ is a finite set and $%
\left\vert \mathcal{G}(x,y)\right\vert $ is invertible in $K,$ for all $%
x,y\in \mathcal{G}_{0}.$
\end{corollary}

\begin{proof}
Let us first assume that $K\left[ \mathcal{G}\right] $ is separable. Since
any separable linear category is locally finite it follows that $\mathcal{G}%
(x,y)$ is a finite set, for any $x,y\in \mathcal{G}_{0}.$ Let $\left(
a_{x}^{y}\right) _{x,y\in \mathcal{G}_{0}}\ $be a family $\ $which satisfies
relations (a)-(d) in Theorem \ref{pr:caracterizare}.(5). We fix $x$ and $y\ $%
in $\mathcal{G}_{0}.$ Hence $a_{x}^{y}\in K\left[ \mathcal{G}\right]
(y,x)\otimes K\left[ \mathcal{G}\right] (x,y).$ Note that $\left\{ g\otimes
h\mid g\in \mathcal{G}(y,x)\text{ and }h\in \mathcal{G}(x,y)\right\} $ is a
basis on $K\left[ \mathcal{G}\right] (y,x)\otimes K\left[ \mathcal{G}\right]
(x,y).$ Thus%
\begin{equation*}
a_{x}^{y}=\sum_{\substack{ g\in \mathcal{G}(y,x)  \\ h\in \mathcal{G}(x,y)}}%
\alpha _{g,h}g\otimes h,
\end{equation*}%
where $\alpha _{g,h}$ is a certain element in $K,$ for every $g\in \mathcal{G%
}(y,x)$ and $h\in \mathcal{G}(x,y).$ Taking into account that $\left(
a_{x}^{y}\right) _{x,y\in \mathcal{G}_{0}}$ satisfies (c), it follows easily
that%
\begin{equation}
\sum_{g\in \mathcal{G}(y,x)}\alpha _{g,g^{-1}}=1.  \label{eq:agg}
\end{equation}%
On the other hand, since $\left( a_{x}^{y}\right) _{x,y\in \mathcal{G}_{0}}$
satisfies (d), for every $f\in \mathcal{G}(x,x)$ we have $f\rhd
a_{x}^{y}=a_{x}^{y}\lhd f.$ $\ $It follows%
\begin{equation*}
\sum_{\substack{ g\in \mathcal{G}(y,x)  \\ h\in \mathcal{G}(x,y)}}\alpha
_{g,h}\left( f\circ g\right) \otimes h=\sum_{\substack{ g\in \mathcal{G}%
(y,x)  \\ h\in \mathcal{G}(x,y)}}\alpha _{g,h}g\otimes \left( h\circ
f\right) .
\end{equation*}%
Hence $\alpha _{f^{-1}\circ u,v}=\alpha _{u,v\circ f^{-1}},$ for any $u\in
\mathcal{G}(y,x)\ $and $v\in \mathcal{G}(x,y).$ We fix $g_{0}\in \mathcal{G}%
(y,x).$ Thus, by taking $f:=u_{0}\circ u^{-1}$ in the above identity, we get%
\begin{equation*}
\alpha _{u,u^{-1}}=\alpha _{(u_{0}\circ u^{-1})^{-1}\circ
u_{0},u^{-1}}=\alpha _{u_{0},u^{-1}\circ (u_{0}\circ u^{-1})^{-1}}=\alpha
_{u_{0},u_{0}^{-1}}.
\end{equation*}%
In conclusion the element $\alpha _{u,u^{-1}}$ does not deppend on $u\in
\mathcal{G}(y,x).$ By (\ref{eq:agg}), we get $\left\vert \mathcal{G}%
(x,y)\right\vert \alpha _{u_{0},u_{0}^{-1}}=1,$ so $\left\vert \mathcal{G}%
(x,y)\right\vert $ is invertible in $K.$

Conversely, let us assume that $\left\vert \mathcal{G}(x,y)\right\vert
<\infty $ and that $\left\vert \mathcal{G}(x,y)\right\vert $ is invertible
in $K.$ It is easy to that the elements%
\begin{equation*}
a_{x}^{y}:=\frac{1}{\left\vert \mathcal{G}(x,y)\right\vert }\sum_{g\in
\mathcal{G}(y,x)}g\otimes g^{-1}
\end{equation*}%
define a family which satisfies the properties (a)-(d) in Theorem \ref%
{pr:caracterizare}.(5), so $K\left[ \mathcal{G}\right] $ is separable.
\end{proof}

Recall that a category $\mathcal{A}$ is said to be skeletal if its only
isomorphisms are automorphisms. A skeletal category $\mathcal{A}$ is called
a delta category if the only endomorphisms in $\mathcal{A}$ are the
identities, cf. \cite[p. 83]{M2}. If $\mathcal{A}$ is a delta category then
there is a partial order relation $\leq $ on $\mathcal{A}_{0}$ such that $%
\mathcal{A}(x,y)\neq \varnothing $ if, and only if, $x\leq y.$ Note that any
poset (regarded as a category) is a delta category. Discrete categories are,
of course, examples of posets (with respect to the trivial order relation).

\begin{corollary}
Let $\mathcal{A}$ be a delta category. Then $K\left[ \mathcal{A}\right] $ is
separable if, and only if $\mathcal{A}$ is a discrete category.
\end{corollary}

\begin{proof}
Clearly, for a discrete category $\mathcal{A},$ the $K$-liniarization $K%
\left[ \mathcal{A}\right] $ is separable. Indeed, the elements%
\begin{equation*}
a_{x}^{y}:=\left\{
\begin{array}{cc}
0, & x\neq y; \\
1_{x}\otimes 1_{x}, & x=y;%
\end{array}%
\right.
\end{equation*}%
define a family $(a_{x}^{y})_{x,y\in }\mathcal{A}_{0}$ which satisfies the
properties (a)-(d) in Theorem \ref{pr:caracterizare}.(5). Conversely, let as
assume that $K\left[ \mathcal{A}\right] $ is separable. We have to prove
that $\mathcal{A}(x,z)=\varnothing ,$ for any $x<z.$ Let $%
(a_{x}^{y})_{x,y\in }\mathcal{A}_{0}$ which satisfies the properties
(a)-(d). Since $\mathcal{A}$ is a delta category it follows that either $K%
\left[ \mathcal{A}\right] (x,y)=0$ or $K\left[ \mathcal{A}\right] (y,x)=0,$
provided that $x\neq y.$ Therefore $a_{x}^{y}=0,$ for any $x$ and $y$ such
that $x\neq y.$ Let us suppose that there is $f:x\rightarrow z$ in $K\left[
\mathcal{A}\right] ,$ with $x<z.$ Since $f\rhd a_{x}^{x}=a_{z}^{x}\lhd
f=0\lhd f$, we deduce that $f\rhd a_{x}^{y}=0$ for any $x,y\in \mathcal{A}%
_{0}.$ Hence%
\begin{equation*}
f=f\rhd 1_{x}=\sum_{y\in \mathcal{A}_{0}}\mathrm{comp}(f\rhd a_{x}^{y})=0.
\end{equation*}%
It follows that $K\left[ \mathcal{A}\right] (x,z)=0,$ for all $x<z.$
Therefore $\mathcal{A}(x,z)=\varnothing .$
\end{proof}

\begin{remark}
For a different proof of the above corollary see \cite[Proposition 33.1]{M2}.
\end{remark}

\begin{proposition}
If $\mathcal{C}$ is separable then any left $\mathcal{C}$-module $M$ is
projective.
\end{proposition}

\begin{proof}
Since $\mathcal{C}$ is separable there is a family $(a_{x}^{y})_{x,y\in
\mathcal{C}_{0}}$ as in Theorem \ref{pr:caracterizare}.(5). It is sufficient
to prove that the canonical morphism of left $\mathcal{C}$-modules $\varphi :%
\mathcal{C}\otimes M\rightarrow M$ has a section in $\mathcal{C}$-$\mathrm{%
Mod}$. Let
\begin{equation*}
\varphi _{x}:\bigoplus_{y\in \mathcal{C}_{0}}\mathcal{C}(y,x)\otimes
{}_{y}M\rightarrow {}_{x}M
\end{equation*}%
be the corresponding component of degree $x$. We define
\begin{equation*}
\psi _{x}^{y}:{}_{x}M\rightarrow \mathcal{C}(y,x)\otimes {}_{y}M,\qquad \psi
_{x}^{y}(m)=\sum_{i=1}^{n_{x,y}}f_{y,x}^{i}\otimes \left( g_{x,y}^{i}\rhd
m\right) ,
\end{equation*}%
where the elements $f_{y,x}^{i}$ and $g_{x,y}^{i}$ define $a_{x}^{y}\in
\mathcal{C}(y,x)\otimes \mathcal{C}(x,y)$ as in relation (\ref{eq:fxy}).
Since the family $(a_{x}^{y})_{y\in \mathcal{C}_{0}}\ $is of finite support,
it follows that the family $(\psi _{x}^{y}(m))_{y\in \mathcal{C}_{0}}$ is
also of finite support. Thus it makes sense to define $\psi
_{x}:{}_{x}M\rightarrow \bigoplus_{y\in \mathcal{C}_{0}}\left( \mathcal{C}%
(y,x)\otimes {}_{y}M\right) $ by $\psi _{x}(m)=(\psi _{x}^{y}(m))_{y\in
\mathcal{C}_{0}}.$ For $m\in {}_{x}M$ we get
\begin{eqnarray*}
(\varphi _{x}\circ \psi _{x})(m) &=&\sum_{y\in \mathcal{C}%
_{0}}\sum_{i=1}^{n_{x,y}}f_{y,x}^{i}\rhd (g_{x,y}^{i}\rhd m)=\sum_{y\in
\mathcal{C}_{0}}\sum_{i=1}^{n_{x,y}}(f_{y,x}^{i}\circ g_{x,y}^{i})\rhd m \\
&=&\left( \sum_{y\in \mathcal{C}_{0}}\mathrm{comp}(a_{x}^{y})\right) \rhd
m=1_{x}\rhd m=m.
\end{eqnarray*}%
Note that the first equality follows by the definition of the maps $\varphi
_{x}$ and $\psi _{x}$. The second and the last equalities are consequences
of the definition of $\mathcal{C}$-modules, while for the third relation we
used the fact that the family $(a_{x}^{y})_{x,y\in \mathcal{C}_{0}}$
satisfies property (c) in Theorem \ref{pr:caracterizare}.(5). Summarizing,
we have proved that $\psi :=(\psi _{x})_{x\in \mathcal{C}_{0}}$ is a section
of $\varphi .$ It remains to show that $\psi $ is a morphism of left $%
\mathcal{C}$-modules. Let $f\in \mathcal{C}(z,x)$ and $m\in {}_{z}M,$ where $%
x$ and $z$ are given objects in $\mathcal{C}_{0}.$ We have
\begin{eqnarray*}
\psi _{x}(f\rhd m) &=&(\psi _{x}^{y}(f\rhd m))_{y\in \mathcal{C}_{0}}=\left(
\sum_{i=1}^{n_{x,y}}f_{y,x}^{i}\otimes \left[ (g_{x,y}^{i}\circ f)\rhd m%
\right] \right) _{y\in \mathcal{C}_{0}} \\
&=&\left( \sum_{j=1}^{n_{y,z}}(f\circ f_{y,z}^{j})\otimes \left(
g_{z,y}^{j}\rhd m\right) \right) _{y\in \mathcal{C}_{0}}=\left( f\rhd
\sum_{j=1}^{n_{y,z}}f_{y,z}^{j}\otimes \left( g_{z,y}^{j}\rhd m\right)
\right) _{y\in \mathcal{C}_{0}}\hspace*{\fill} \\
&=&\left( f\rhd \psi _{z}^{y}(m)\right) _{y\in \mathcal{C}_{0}}=f\rhd \left(
\psi _{z}^{y}(m)\right) _{y\in \mathcal{C}_{0}}=f\rhd \psi _{z}(m).
\end{eqnarray*}%
The first and the second identities follow by the definition of $\psi _{x}$
and $\psi _{x}^{y},$ respectively. For the third equality one uses property
(d) in Theorem \ref{pr:caracterizare}.(5), while the fourth one is obtained
from the definition of the action on $\mathcal{C}\otimes M$. Finally, the
fifth and the last relations are consequences of the definition of $\psi
_{z}^{y}$ and $\psi _{z}$, while for the sixth identity one uses the
definition of the $\mathcal{C}$-module structure on $M:=\bigoplus_{y\in
\mathcal{C}_{0}}{}_{y}M.$
\end{proof}

\end{document}